\documentclass[12pt]{amsart}
\usepackage{amssymb,amsmath,amsthm}
\usepackage{mathtools}
\usepackage{tikz}
\usepackage{hyperref}
\usepackage{cleveref}
\usepackage{todonotes}
\newcommand{\old}[1]{}
\usepackage[left=1.25in, right=1.25in]{geometry}

\theoremstyle{plain}
\newtheorem{thm}{Theorem}[section]
\newtheorem{lem}[thm]{Lemma}

\newtheorem{cor}[thm]{Corollary}

\newtheorem{prop}[thm]{Proposition}
\theoremstyle{definition}
\newtheorem{defn}[thm]{Definition}
\newtheorem{ex}[thm]{Example}

\Crefname{thm}{Theorem}{Theorems}
\Crefname{lem}{Lemma}{Lemmas}

\theoremstyle{remark}
\newtheorem{remark}[thm]{Remark}

\renewcommand{\SS}{\mathbb{S}}

\DeclareMathOperator{\conv}{conv}
\DeclareMathOperator{\psj}{psj}

\DeclareMathOperator{\sink}{sink}
\DeclareMathOperator{\srce}{source}
\DeclareMathOperator{\im}{im}

\def\reals{{\mathbb R}}
\def\integers{{\mathbb Z}}

\tikzstyle{BlueLine}=[line width=0.3mm,color=blue,text=black]
\tikzstyle{BluePoly}=[BlueLine,fill=blue!20]
\tikzstyle{RedLine}=[line width=0.3mm,color=red,text=black]
\tikzstyle{RedPoly}=[RedLine,fill=red!20]
\tikzstyle{GreenLine}=[thick,color=black!30!green,text=black]

\title{Poset topology, moves, and Bruhat interval polytope lattices}

\author{Christian Gaetz}
\address[Gaetz]{Department of Mathematics, University of California, Berkeley, CA 94720}
\email{gaetz@berkeley.edu}

\author{Patricia Hersh}
\address[Hersh]{Department of Mathematics, University of Oregon, Eugene, OR 97403}
\email{plhersh@uoregon.edu}

\thanks{P.H. was supported by NSF grant DMS-1953931.  This research was also supported by NSF-RTG grant DMS-2039316. C.G. was partially supported by NSF grant DMS-2452032 and by a
travel grant from the Simons Foundation.}

\begin{document}

\begin{abstract}
We study the poset topology of lattices arising from orientations of 1-skeleta of \emph{directionally simple} polytopes, with Bruhat interval polytopes $Q_{e,w}$ as our main example. We show that the order complex $\Delta((u,v)_w)$ of an interval therein is homotopy equivalent to a sphere if $Q_{u,v}$ is a face of $Q_{e,w}$ and is otherwise contractible. This significantly generalizes the known case of the permutahedron. We also show that saturated chains from $u$ to $v$ in such lattices are connected, and in fact highly connected, under flipping across 2-faces. When $w$ is a Grassmannian permutation, this implies a strengthening of the restriction of Postnikov's move-equivalence theorem to the class of BCFW bridge decomposable plabic graphs.
\end{abstract}

\keywords{Bruhat interval polytope, 1-skeleton poset, BCFW bridge decomposition, poset topology, directionally simple, flip graph}
\subjclass[2020]{05E45, 06A07} 

\maketitle

\section{Introduction}\label{sec:intro}

\emph{Bruhat interval polytopes} were introduced in \cite{KW} and further studied in \cite{TW}, motivated by their connections to the Toda lattice and to the moment map on the totally positive part of the flag variety. Recently, these polytopes have attracted significant attention as the moment polytopes of generic torus orbit closures in Schubert varieties \cite{lee-masuda, lee-masuda-park-song}, as distinguished flag positroid polytopes \cite{boretsky-eur-williams}, and for their connection to Mirkovi\'{c}--Vilonen polytopes \cite{sanchez}.

It is natural to study the \emph{1-skeleton poset} $P_w$ of a Bruhat interval polytope $Q_w=Q_{e,w}$, obtained by orienting the 1-skeleton of $Q_w$ with respect to a certain cost vector $\boldsymbol{\rho}$. This is a partial order on the lower Bruhat interval $[e,w]$ which is intermediate in strength between the weak order and the strong (Bruhat) order. The order complexes of intervals in both the weak and strong orders are known \cite{bjorner-bruhat, bj, BW-bruhat-order} to be homotopy equivalent to balls or spheres; indeed these are fundamental examples in poset topology. These results have also found many applications to the study of the weak and strong orders, and in particular they allow for the easy determination of their M\"{o}bius functions, which have long (see e.g. \cite{verma}) been of interest. In our first main result, we show that these results extend to intervals in the posets $P_w$. 

We write $[u,v]_w$ for the closed interval between $u$ and $v$ in the poset $P_w$ and $A_w(u,v)$ for the set of atoms of this interval; we write $(u,v)_w$ for the open interval. 

\begin{thm}
\label{thm:intro-topology}
The order complex $\Delta((u,v)_w)$ of any open interval in $P_w$ is homotopy equivalent 
to a sphere $\SS^{|A_w(u,v)|-2}$ if $Q_{u,v}$ is a face of $Q_w$, and is otherwise contractible. Thus the M\"{o}bius function takes values $\mu_{P_w}(u,v) = (-1)^{|A_w(u,v)|}$ or $\mu_{P_w}(u,v)=0$ accordingly.
\end{thm}

Since $i$-faces of $Q_w$ correspond to $i$-dimensional torus orbits in a generic torus orbit closure $Y_w$ in the Schubert variety $X_w$, \Cref{thm:intro-topology} gives a new combinatorial way of identifying the higher-dimensional orbits given only the $0$- and $1$-dimensional orbits. \Cref{thm:intro-topology} also answers a question\footnote{Personal communication at FPSAC 2023.} of Richard Stanley, who asked for a description of the M\"{o}bius function of $P_w$. 

Our approach to \Cref{thm:intro-topology} extends the techniques developed by the second named author in \cite{He}, but requires new arguments. It was proven there that a \emph{simple} polytope whose 1-skeleton poset is the Hasse diagram of a lattice satisfies the conclusions of \Cref{thm:intro-topology}. The first named author proved in \cite{Ga} that the posets $P_w$ are indeed lattices. The polytopes $Q_w$ are not in general simple, but they do have the weaker property of being \emph{directionally simple} \cite{Ga}. We give a new argument to show in \Cref{thm:join-equals-pseudo} that these conditions suffice to imply that the lattice-theoretic join agrees with a convex-geometric \emph{pseudojoin} operation. We use this together with a face non-revisiting property of $Q_w$ to prove \Cref{thm:intro-topology}. 

\begin{remark}
Our proof of \Cref{thm:intro-topology} could be adapted without significant changes to apply to any directionally simple polytope $Q$ with facial orientation $\mathcal{O}$ such that $\mathcal{O}$ is the Hasse diagram of a lattice and such that $Q$ satisfies the face non-revisiting property (see \Cref{prop:nonrevisit-property} for this notion of face non-revisiting  and see Section ~\ref{sec:prelim} for the notions of directional simplicity and of facial orientation). It would be interesting to find other examples of such polytopes.  
\end{remark}

Our second main result concerns maximal chains in $[u,v]_w$ (that is, facets of $\Delta((u,v)_w)$) or more generally in any lattice which is the 1-skeleton poset of a directionally simple polytope. If a maximal chain $\gamma_1$ enters some 2-dimensional face $F$ at its unique source (with respect to our fixed 1-skeleton orientation), continues along part of $\partial F$, and then 
exits at the unique sink of $F$, we can perform a \emph{flip} to produce a new maximal chain $\gamma_2$ which goes the other way around $\partial F$. Two maximal chains are \emph{flip-equivalent} if they can be connected by a sequence of such flips. We study the graph $\mathcal{M}(u,v)$ whose vertices are these maximal chains and whose edges correspond to such flips.

\begin{thm}
\label{thm:intro-moves}
Let $\mathcal{O}$ be a facial orientation (see \Cref{def:facial-orientation}) and suppose that a polytope $Q$ is $\mathcal{O}$-directionally simple such that $\mathcal{O}$ is the Hasse diagram of a lattice $L$. Suppose that $[u,v]_L$ has $a$ atoms. Then $\mathcal{M}(u,v)$ is connected, and if $a \geq 2$ is moreover $(a-1)$-connected.
\end{thm}   

\Cref{thm:intro-moves} is surprisingly powerful. For example, when $Q=Q_{w_0}$ is the \emph{permutahedron} the lattice is the (right) weak order on the symmetric group. Maximal chains in weak order correspond to \emph{reduced words} of $w_0$, and a flip across a 2-face $F$ corresponds to a \emph{commutation move} or \emph{braid move} according to whether $F$ is a square or hexagon. Thus in this case \Cref{thm:intro-moves} recovers Matsumoto's Theorem \cite{matsumoto, tits} that all reduced words are connected under commutation or braid moves.

\Cref{thm:intro-moves} extends a similar result of Athanasiadis, Edelman, and Reiner who work in the setting of simple polytopes (and with the assumption that $u=\hat{0}, v=\hat{1}$) \cite{AER}. Our generalization to directionally simple polytopes is important because $Q_w$ is almost never simple for $w$ a Grassmannian permutation, and flip-equivalence on these Bruhat interval polytopes is of particular interest: if $w$ is Grassmannian, then, by a result of Williams \cite{wi}, $Q_w$ is isomorphic to a \emph{bridge polytope} and maximal chains in $P_w$ correspond to \emph{BCFW bridge decompositions}.
 
BCFW bridge decompositions are physically motivated ways of building up reduced \emph{plabic graphs} with trip permutation $w$ \cite{arkani-hamed-book}. Flips across a 2-face now correspond to applications of Postnikov's moves \cite{postnikov} to the resulting plabic graphs (see \cite[Thm.~5.3]{wi}). A fundamental theorem of Postnikov \cite{postnikov} says that any two reduced plabic graphs with the same trip permutation are move-equivalent. We obtain the restriction of this result to the class of BCFW decomposable plabic graphs as an easy corollary of \Cref{thm:intro-moves}. Our result has the advantage of guaranteeing that two BCFW decomposable plabic graphs can be connected by moves, with all intermediate graphs also BCFW decomposable\footnote{We thank Melissa Sherman-Bennett for pointing out that there exist reduced plabic graphs which are not BCFW decomposable.}. Furthermore, we are not aware of prior results on the higher connectivity of the plabic moves.

\begin{cor}
\label{cor:BCFW}
Let $w \in S_n$ be Grassmannian. Then any two BCFW bridge decomposable plabic graphs with trip permutation $w$ can be connected by a sequence of moves, without leaving the class of BCFW decomposable plabic graphs. Moreover, the move graph is $(a-1)$-connected, where $a$ is the number of distinct simple reflections appearing in a reduced word for $w$.
\end{cor}

It would be reasonable to hope for \Cref{thm:intro-moves} to hold  for more general polytopes, in light of Balinski's theorem. After all, when all of the  monotone paths in a polytope are coherent  then results of \cite{BKS} imply that  the graph $\mathcal{M}(u,v)$  is  the graph of a $(d-1)$-polytope, specifically it is   the graph of a monotone path polytope (recall e.g. from \cite{Zi}  that Balinski's Theorem states that the graph of any $(d-1)$-polytope is $(d-1)$-connected). This line of reasoning does work, for example, for the  zonotopes of \cite{EJLM} having the all-coherence property and the polymatroid polytopes which were proven in  \cite{BS} to have only coherent monotone paths. However, it is shown in \cite{AER} that the higher connectivity of \Cref{thm:intro-moves} fails for general polytopes.  It is an interesting question for future study to determine exactly how much further \Cref{thm:intro-moves} can be pushed.  

\section{Preliminaries}\label{sec:prelim}

\subsection{Bruhat interval polytopes}

We refer the reader to \cite{BB} for background on Coxeter groups and Bruhat order. We view the symmetric group $S_n$ as a Coxeter group, with adjacent transpositions as the simple reflections. We write $\preceq$ for the Bruhat order on $S_n$, and $\leq_R$ for the right weak order.

\begin{defn}[Kodama--Williams \cite{KW}]
For a permutation $w \in S_n$, denote by $\mathbf{w}$ the vector $(w^{-1}(1),\ldots, w^{-1}(n)) \in \mathbb{R}^n$. For $w' \preceq w \in S_n$, define the corresponding \emph{Bruhat interval polytope} as the convex hull $Q_{w',w}\coloneqq \conv \{ \mathbf{u} \mid w' \preceq u \preceq w\}$. If $w'=e$ is the identity permutation, we write $Q_w$ for $Q_{e,w}$. See \Cref{fig:Bruhat-interval-polytope-picture} for an example.
\end{defn}

\begin{defn}
Let $P_w$ denote the partially ordered set $([e,w], \leq_w)$ with cover relations $u \lessdot_w v$ whenever $\ell(v)=\ell(u)+1$ (where $\ell$ denotes Coxeter length) and there is an edge of $Q_w$ connecting $\mathbf{u}$ and $\mathbf{v}$. We often conflate elements $u$ of $P_w$ and vertices $\mathbf{u}$ of $Q_w$ when no confusion can result.
\end{defn}

 \begin{figure} 
\begin{tikzpicture}[scale=0.6]

\draw (7,1)--(2,4)--(0,7)--(1,10)--(7,13); 
\draw (7,1)--(7,4)--(5,7)--(7,10)--(7,13);
\draw(5,7)--(1,10);
\draw(7,4)--(8,7)--(7,10);
\draw(8,7)--(12,10)--(7,13);
\draw(7,1)--(13,4)--(14,7)--(12,10); 
\draw [dashed] (2,4)--(10,7)--(9,10)--(7,13);
\draw [dashed] (0,7)--(9,10);
\draw [dashed] (13,4)--(10,7);
\draw [dashed] (14,7)--(9,10);

\node [below] at (7,1){\scriptsize $1234$}; 
\node [below left] at (2,4){\scriptsize $2134$}; 
\node [below left] at (7,4){\scriptsize $1324$}; 
\node [below right] at (13,4){\scriptsize $1243$}; 
\node [below left] at (0,7){\scriptsize $2314$}; 
\node [below left] at (5,7){\scriptsize $3124$}; 
\node [left] at (8,7){\scriptsize $1342$}; 
\node [above right] at (10,6.8){\scriptsize $2143$}; 
\node [below right] at (14,7){\scriptsize $1423$}; 
\node [above left] at (1,10){\scriptsize $3214$}; 
\node [above left] at (7,10){\scriptsize $3142$}; 
\node [above right] at (9,10){\scriptsize $2413$}; 
\node [above right] at (12,10){\scriptsize $1432$}; 
\node [above] at (7,13){\scriptsize $3412$};

\draw [fill, ] (7,1) circle [radius=0.05]; 
\draw [fill] (2,4) circle [radius=0.05]; 
\draw [fill] (7,4) circle [radius=0.05]; 
\draw [fill] (13,4) circle [radius=0.05]; 
\draw [fill] (0,7) circle [radius=0.05]; 
\draw [fill] (5,7) circle [radius=0.05]; 
\draw [fill] (8,7) circle [radius=0.05]; 
\draw [fill] (10,7) circle [radius=0.05];
\draw [fill] (14,7) circle [radius=0.05];
\draw [fill, ] (1,10) circle [radius=0.05]; 
\draw [fill] (7,10) circle [radius=0.05]; 
\draw [fill] (9,10) circle [radius=0.05]; 
\draw [fill] (12,10) circle [radius=0.05]; 
\draw [fill] (7,13) circle [radius=0.05]; 

\end{tikzpicture}
\caption{The Bruhat interval polytope $Q_{3412}$.}
\label{fig:Bruhat-interval-polytope-picture}
\end{figure}

A key property of Bruhat interval polytopes is the following.   

\begin{thm}[Tsukerman-Williams]\label{thm:faces-are-Bruhat-interval-polytopes}
Every face of a Bruhat interval polytyope is a Bruhat interval polytope.
\end{thm}

\Cref{thm:faces-are-Bruhat-interval-polytopes} implies in particular that every cover relation in $P_w$ is a cover relation in Bruhat order. This implies in turn that Bruhat interval polytopes are \emph{generalized permutahedra} in the sense of Postnikov \cite{Postnikov-beyond}.   

The poset $P_w$ can alternatively be described as the 1-skeleton poset of the polytope $Q_w$ with respect to a cost vector. Such posets are the subject of \cite{He}.

\begin{defn}
\label{def:facial-orientation}
We say an acyclic orientation $\mathcal{O}$ of the 1-skeleton of a polytope $Q$ is a \emph{facial orientation} if the restriction of $\mathcal{O}$ to each face $F$ of $Q$ has a unique source and a unique sink; we write $\srce(F)$ and $\sink(F)$ for these vertices. The \emph{1-skeleton poset} $P(Q,\mathcal{O})$, a partial order on the vertices of $Q$, is defined as the transitive closure of $\mathcal{O}$. In this paper, we will be most interested in cases where $\mathcal{O}$ is the Hasse diagram of $P(Q,\mathcal{O})$ (that is, when $\mathcal{O}$ is transitively reduced); this is the \emph{Hasse diagram property} of \cite{He}.
\end{defn}

\emph{Generic cost vectors} give an easy source of facial orientations. For $Q \subset \reals^d$, a vector $\mathbf{c}\in \reals^d$ is \emph{generic} if no edge of $Q$ is orthogonal to $\mathbf{c}$. Given such a vector, we obtain a facial orientation $\mathcal{O}_{\mathbf{c}}$ by orienting each edge according to increasing inner product with $\mathbf{c}$; in this case we may write $P(Q,\mathbf{c})$ for $P(Q,\mathcal{O}_{\mathbf{c}})$.

\begin{ex}
Consider the permutahedron $Q_{w_0}$ for $S_n$. It is observed in 
\cite{He} that for cost vector ${\boldsymbol \rho} \coloneqq (n,n-1,\dots ,1)$ the poset $P(Q_{w_0},\boldsymbol{\rho})$ is the right weak order on $S_n$ and that $\mathcal{O}_{\boldsymbol{\rho}}$ has the Hasse diagram property. It is further observed in \cite{Ga} that for any Bruhat interval polytope $Q_w$ with $w \in S_n$ and this same cost vector ${\boldsymbol \rho}$ that $\mathcal{O}_{\boldsymbol{\rho}}$ is the Hasse diagram of the poset $P_w \coloneqq P(Q_w,\boldsymbol{\rho})$.
\end{ex}

Recall that a $d$-polytope $Q$ is {\it simple} if each vertex is incident to exactly $d$ edges, or, equivalently, if, for each vertex $v\in Q$, each subset of the edges incident to $v$ spans a distinct face. Most of the results in \cite{He} were proven for simple polytopes. In this paper, we work in the more general setting of \emph{directionally simple} polytopes. We work with a more general notion of directional simplicity than that introduced in \cite{Ga}, which only allowed for orientations induced by a generic cost vector.

\begin{defn}
We say $Q$ is \emph{$\mathcal{O}$-directionally simple} with respect to a facial orientation $\mathcal{O}$ if for each vertex $v$ of $Q$, each subset $E$ of the edges outgoing from $v$ (with respect to $\mathcal{O}$) spans a distinct face of $Q$. We write $F_E$ for this face.
\end{defn}

Directional simplicity is a weaker condition than simplicity since the existence of $F_E$ is required only for subsets of edges outgoing from $v$, not for all subsets of edges incident to $v$.

\begin{thm}[Theorem 5.3, \cite{Ga}]
\label{thm:bip-is-directionally-simple}
For any $w \in S_n$, the Bruhat interval polytope $Q_w$ is directionally simple with respect to the cost vector $\boldsymbol{\rho}$.
\end{thm}

It will also be very important for us that the posets $P_w$ are lattices.

\begin{thm}[Theorem 4.5, \cite{Ga}]
\label{thm:bip-is-lattice}
For any $w \in S_n$, the poset $P_w$ is a lattice.
\end{thm}

\begin{remark}
The more general polytopes $Q_{w',w}$ do not always have the lattice or directional simplicity properties (see \cite[\S1.3.3]{Ga}), so our results do not necessarily apply to these.
\end{remark}

\subsection{Poset topology}

We briefly review some background on poset topology.

\begin{defn}
The {\it M\"obius function} of a finite bounded partially ordered set $(P,\leq)$ is the function $\mu_P : P\times P \rightarrow \integers $ satisfying 
$$\mu_P(u,v) = -\sum_{u\le z < v} \mu_P(u,z),$$
for $u\ne v$, and $\mu_P(u,u)=1$, for all $u \in P$.
\end{defn}

\begin{defn}
The {\it order complex}, denoted $\Delta (P)$, of a poset $P$ is the abstract simplicial complex whose $i$-faces are the chains $v_0 < \cdots < v_i$ of $P$.  
\end{defn}

We will use the following well known fact due to Philip Hall.

\begin{prop}
\label{prop:philip-hall}
Let $P$ be a finite bounded poset and let $u \leq_P v$. Then
\[
\mu_P(u,v) = \tilde{\chi }(\Delta((u,v)_P)),
\]
where $\tilde{\chi}$ denotes reduced Euler characteristic.
\end{prop}

Recall from \cite{Qu} the Quillen fiber lemma:  

\begin{lem}[Quillen fiber lemma]
Suppose that $f:P\rightarrow P'$ is an order-preserving map such that for each $p'\in P'$ the subposet $P^{\le p'} \coloneqq \{ p\in P| f(p)\le p' \} $ has contractible order complex. Then $\Delta(P)$ and $\Delta(P')$ are homotopy equivalent.
\end{lem}

Since a poset and its dual have the same order complex, an immediate consequence is what we will call the ``dual Quillen fiber lemma'':

\begin{cor}[dual Quillen fiber lemma]
\label{cor:dual-quillen}
Suppose that $f:P\rightarrow Q$ is an order-preserving map such that for each $p'\in P'$ the subposet $P^{\ge p'} \coloneqq \{ p\in P | f(p) \ge p'\}$ has contractible order complex. Then $\Delta (P)$ and $\Delta (P')$ are homotopy equivalent.
\end{cor}

\subsection{Topology of simple 1-skeleton lattices}
The following result of \cite{He} helped motivate the present work.   

\begin{thm}[Theorem 5.13, \cite{He}]
\label{thm:simple-lattice-homotopy}
Let $Q$ be a simple polytope with facial orientation $\mathcal{O}$ such that $L \coloneqq P(Q,\mathcal{O})$ is a lattice with Hasse diagram $\mathcal{O}$. Then any open interval $(u,v)_L$ has order complex homotopy equivalent to a sphere $\SS^{a-2}$, where $a$ is the number of atoms in $[u,v]_L$, if $v$ equals the join of the atoms of $[u,v]$, and is otherwise contractible.
\end{thm} 

It is natural to try to generalize this result to directionally simple polytopes.
We will also need the next notion which was introduced in  \cite{He} for simple polytopes but makes equally good sense for directionally simple polytopes.

\begin{defn}
Suppose that a polytope $Q$ is $\mathcal{O}$-directionally simple where $\mathcal{O}$ is the Hasse diagram of $P(Q,\mathcal{O})$. Then the {\it pseudo-join} $\psj(S)$ of a collection $S$ of upper covers of an element $u$ in $P(Q,\mathcal{O})$ is the unique sink of the unique $|S|$-face $F_S$ containing $u$ and the elements of $S$. The face $F_S$ exists and is unique by directional simplicity and has a unique sink since $\mathcal{O}$ is a facial orientation.
\end{defn}

\begin{thm}[Theorem 4.7, \cite{He}]
Let $Q$ be a simple polytope with facial orientation $\mathcal{O}$ such that $\mathcal{O}$ is the Hasse diagram of a lattice $L$.  Then for any collection $S$ of atoms of an interval $[u,v]_L$, we have $\bigvee S = \psj(S)$.  
\end{thm}

\section{Poset topology of Bruhat interval polytopes}

In this section we prove \Cref{thm:intro-topology}, our first main theorem.

Recall that we write $\preceq$ for Bruhat order and $\leq_w$ for the order on $P_w$. We write $(u,v)_B$ and $[u,v]_B$ for open and closed intervals in Bruhat order and $(u,v)_w$ and $[u,v]_w$ for open and closed intervals in $P_w$. We also fix the notation $\mathcal{O}_w$ for the orientation of the 1-skeleton of $Q_w$ induced by the cost vector $\boldsymbol{\rho}$; this is the Hasse diagram of $P_w$. Finally, we write $\wedge_w$ and $\vee_w$ for the meet and join operators in the lattice $P_w$.

Combining \Cref{thm:bip-is-lattice,thm:simple-lattice-homotopy}, it is easy to conclude that \Cref{thm:intro-topology} holds whenever the Bruhat interval polytope $Q_w$ is simple. Our goal is to extend this result to all $Q_w$. 

We first establish a face non-revisiting property for the $Q_w$.

\begin{prop}\label{prop:nonrevisit-property}
Let $v_0 \to v_1 \to \cdots \to v_k$ be a directed path in $\mathcal{O}_w$ and suppose that $v_0$ and $v_k$ both lie in some face $F$ of $Q_w$. Then $v_i$ lies in $F$ for all $i=0,\ldots,k$.
\end{prop}

\begin{proof}
By \Cref{thm:faces-are-Bruhat-interval-polytopes}, the vertices of the face $F$ are the elements $[u,v]_B$ of a Bruhat interval. Note that $[u,v]_w \subset [u,v]_B$ since every cover relation in $P_w$  is a cover relation in Bruhat order. Thus we have $v_i \in [v_0,v_k]_w \subset [u,v]_w \subset [u,v]_B$ for all $i$, so all $v_i$ lie in $F$.
\end{proof}

Next we prove that distinct sets of atoms within closed intervals have distinct joins.  This will allow us to identify a subposet of the interval that is isomorphic to a Boolean lattice.

\begin{prop}\label{prop:distinct-joins}
Let $u \leq_w v$ and let $S,S' \subset A_w(u,v)$ be distinct subsets. Then $\bigvee S \neq \bigvee S'$. 
\end{prop}
\begin{proof}
Since $A_w(u,v) \subset A_w(u,w)$, it suffices to prove the claim when $v=w$, so suppose we are in this case. Suppose for the sake of contradiction that $\bigvee S = \bigvee S' = x$ for $S$ and  $S'$ distinct subsets of $A_w(u,v)$.  For any $T \subset A_w(u,v)$, let $F_T$ denote the unique smallest face of $Q_w$ containing $u$ and all elements of $T$. Notice that, by \Cref{prop:nonrevisit-property}, $x \in F_S \cap F_{S'}$, since $\sink(F_S)$ and $\sink(F_{S'})$ are upper bounds in $P_w$ for $S$ and $S'$, respectively. Suppose without loss of generality that $S' \not \subset S$ and consider some $y \in S' \setminus S$.  Then $y\not\in F_S \cap F_{S'}$ by directional simplicity for $Q_w$ (\Cref{thm:bip-is-directionally-simple}).  But by our choice of $y$ and $x$, there is a directed path from $u$ to $y$ to $x$, contradicting \Cref{prop:nonrevisit-property}.
\end{proof}

\begin{remark}
The proof of Proposition ~\ref{prop:distinct-joins} holds more generally for any directionally simple polytope whose directed 1-skeleton is the Hasse diagram of a lattice and satisfies the face non-revisiting property.  
\end{remark}

\begin{prop}\label{prop:surjective-poset-map}
For any $u \leq_w v$ there is a surjective order-preserving map $f:[u,v]_w \rightarrow \mathcal{P}(A_w(u,v))$, where $\mathcal{P}$ denotes the power set (partially ordered by inclusion).
\end{prop}

\begin{proof}
For each  $v'\in [u,v]_w$, define $f(v')=\{a \in A_w(u,v) \mid a \leq_w v'\}$. The map $f$ is clearly order-preserving and is surjective since $f(\bigvee S)=S$ by \Cref{prop:distinct-joins}.  
\end{proof}

We now prove a general theorem about directionally simple polytopes whose 1-skeleton poset is the Hasse diagram of a lattice, namely that joins equal pseudojoins. This generalizes \cite[Thm.~4.7]{He} which was proven for the case of simple polytopes. 

\begin{thm}\label{thm:join-equals-pseudo}
Let $Q$ be $\mathcal{O}$-directionally simple such that $\mathcal{O}$ is the Hasse diagram of a lattice $L = P(Q,\mathcal{O})$. Let $u \in L$ and suppose $u \lessdot a_1,\ldots,a_r$. Then $a_1 \vee \cdots \vee a_r = \psj(a_1,\ldots,a_r)$.
\end{thm}
\begin{proof}
First observe that $a_1\vee \cdots \vee a_r \le \psj(a_1,\dots ,a_r)$ by virtue of $\psj(a_1,\dots ,a_r)$ being an upper bound for $a_1,\dots ,a_r$.  Therefore it suffices to show that $\psj(a_1,\dots ,a_r) \le a_1 \vee \cdots \vee a_r$.

We proceed by induction on $r$, using the base case $r=2$. This case was proven as \cite[Thm.~4.6]{He} (see \cite[Thm.~5.13]{He} for discussion of the case for general facial orientations $\mathcal{O}$) in the setting of simple polytopes, but we note that the proof applies equally well to directionally simple polytopes without modification. So assume now that $r \geq 3$; in this case we need a different proof than in \cite{He}.

Let $U$ denote the set of vertices of the face $F_{a_1,\ldots,a_r}$. We will show that all elements of $U$, including $\psj(a_1,\ldots,a_r)$, lie below $a_1 \vee \cdots \vee a_r$. Let $u=u_0, u_1,\ldots, u_k$ be an ordering of $U$ such that if $u_i < u_j$ then $i<j$. Each element of $U$ lies in some facet of $F_{a_1,\ldots,a_r}$. We will prove by induction on the index $i$ of the vertex $u_i=\srce(G)$ that all vertices of each facet $G$ lie weakly below $a_1 \vee \ldots \vee a_r$. If $i=0$, so $\srce(G)=u_0=u$, then $G=F_{a_1,\ldots,\widehat{a_j},\ldots,a_r}$ for some $j$. By the inductive hypothesis on $r$, we have 
\[
\sink(G)=\psj(a_1,\ldots,\widehat{a_j},\ldots,a_r)=a_1 \vee \cdots \vee \widehat{a_j} \vee \cdots \vee a_r \leq a_1 \vee \cdots \vee a_r.
\]
Thus all vertices of $G$ lie weakly below $a_1 \vee \cdots \vee a_r$.

Now suppose $\srce(G)=u_i$ with $i>0$. Let $b_1,\ldots,b_{r-1}$ be the upper covers of $u_i$, so $G=F_{b_1,\ldots,b_{r-1}}$. Since $\dim(F_{a_1,\ldots,a_r})=r \geq 3$, the edge $\overline{u_ib_1}$ also lies in some other facet $G'\neq G$ of $F_{a_1,\ldots,a_r}$. Let $u_j=\srce(G')$. Note that $u_j \neq u_i$, since otherwise $u_i$ would have at least $r$ upper covers, an impossibility since $u_i \neq u$ and since $Q$ is directionally simple. Thus we have $u_j < u_i$ and, by our choice of ordering, $j<i$. We conclude by induction on $i$ that all vertices of $G'$, and in particular $b_1$, lie weakly below $a_1 \vee \cdots \vee a_r$. The same is true of $b_2,\ldots,b_{r-1}$, and hence $\psj(b_1,\ldots,b_{r-1})= b_1 \vee \cdots \vee b_{r-1} \leq a_1 \vee \cdots \vee a_r$ by the inductive assumption on $r$. We conclude that all vertices of $G$ lie weakly below $a_1 \vee \cdots \vee a_r$, and this completes the proof. 
\end{proof}

\begin{remark}
The proof of \Cref{thm:join-equals-pseudo} refers readers to \cite[Thm.~4.6 and Thm.~5.13]{He} for the $r=2$ base case because it would be a rather lengthy digression to include this argument.  However, one can give the following much shorter proof for $r=2$ if one also assumes the face non-revisiting property. There exists a directed path from $u$ to $a_1\vee a_2$ to  $\psj(a_1,a_2)$ since $\psj(a_1,a_2)$ is an upper bound for $a_1$ and $a_2$.  The non-revisiting property then implies $a_1\vee a_2 \in F_{a_1,a_2}$.  But 
$\psj(a_1,a_2)$ is the only upper bound for $a_1$ and $a_2$ within $F_{a_1,a_2}$, so  $a_1\vee a_2 = \psj(a_1,a_2)$.
\end{remark}

Applying \Cref{thm:bip-is-lattice,thm:bip-is-directionally-simple,thm:join-equals-pseudo}, we have the following corollary.

\begin{cor}\label{cor:Bruhat-join-equals-pseudo}
For any $u \lessdot_w a_1,\ldots,a_r$ we have $\psj(a_1,\ldots,a_r)=a_1 \vee_w \cdots \vee_w a_r$.
\end{cor}

With a little further work, we also derive the following corollary.

\begin{cor}\label{cor:face-description} 
The Bruhat interval polytope $Q_{u,v}$ is a face of $Q_w$ if and only if $v=\bigvee A_w(u,v)$.
\end{cor}

\begin{proof}
If $v=\bigvee A_w(u,v)$, then by \Cref{cor:Bruhat-join-equals-pseudo} we have $v=\psj(A_w(u,v))$. Thus the vertex set of $F_{A_w(u,v)}$ is exactly $[u,v]_w$ by \Cref{prop:nonrevisit-property}. By \Cref{thm:faces-are-Bruhat-interval-polytopes}, we must have $F_{A_w(u,v)}=Q_{u,v}$, so $Q_{u,v}$ is a face.

Otherwise we have $v>_w\bigvee A_w(u,v)=\psj(A_w(u,v))$, so $v$ does not lie in $F_{A_w(u,v)}$. Then $Q_{u,v}$ cannot be a face of $Q_w$, since $F_{A_w(u,v)} \subsetneq Q_{u,v}$ would both be $|A_w(u,v)|$-dimensional faces.
\end{proof}

Recall that a \emph{cone point} in a simplicial complex is a vertex contained in all of its facets.  Any simplicial complex with a cone point is contractible, because the complex can be contracted to the cone point. If a poset $P$ has a unique minimal element $\hat{0}$, then $\Delta(P)$ has a cone point, since all maximal chains in $P$ contain $\hat{0}$, so $\Delta(P)$ is contractible.

We are now ready to prove \Cref{thm:intro-topology}.

\begin{proof}[Proof of \Cref{thm:intro-topology}]
The order-preserving $f$ from \Cref{prop:surjective-poset-map} restricts to a surjection from $(u,v)_w$ either
to $\mathcal{S}_1 \coloneqq \mathcal{P}(A_w(u,v)) \setminus \{\emptyset, A_w(u,v)\}$ or to $\mathcal{S}_2 \coloneqq \mathcal{P}(A_w(u,v)) \setminus \{\emptyset\}$ according to whether $A_w(u,v) \in \im(f|_{(u,v)_w})$, that is, according to whether $\bigvee A_w(u,v) <_w v$.  Notice that, since $P_w$ is a lattice by \Cref{thm:bip-is-lattice}, the dual fiber $(u,v)_{w}^{\ge S}$ 
has a unique minimal element, namely $\bigvee S$, for each $\emptyset \neq S\subseteq A_w(u,v)$.  This is a cone point in $\Delta((u,v)_{w}^{\ge S})$, so this complex is contractible. Thus we may apply apply the dual Quillen Fiber Lemma (\Cref{cor:dual-quillen}) to deduce that $\Delta((u,v)_w)$ is homotopy equivalent to 
$\Delta(\mathcal{S}_1)$ or to $\Delta(\mathcal{S}_2)$.  The complex $\Delta(\mathcal{S}_2)$ is contractible due to having the cone point $A_w(u,v)$. The poset $\mathcal{S}_1$ is the face poset of the boundary of a simplex, so its order complex is the barycentric subdivision of this sphere. Thus $\Delta(\mathcal{S}_1)$ is a sphere $\SS^{|A_w(u,v)|-2}$.
We apply \Cref{cor:face-description} to deduce that we are in the case of a sphere exactly when $Q_{u,v}$ is a face of $Q_w$. Finally, the claim about $\mu_{P_w}$ follows easily by applying \Cref{prop:philip-hall}.
 \end{proof}

\section{Higher connectivity of flip graphs}

In this section we study flip graphs of chains in 1-skeleton posets and prove \Cref{thm:intro-moves}.

\begin{defn}
Let $Q$ be $\mathcal{O}$-directionally simple. Let $\mathcal{M}(Q,\mathcal{O})$ denote the graph whose vertices are the directed paths from source to sink in $\mathcal{O}$ and whose edges connect pairs of such paths differing by a single flip across a 2-face of $Q$. 

When $\mathcal{O}$ is the Hasse diagram of a poset $P$, for each $u \leq_P v$ we define $\mathcal{M}(u,v,\mathcal{O})$ as 
the graph whose vertices are the saturated chains from $u$ to $v$ and whose edges connect pairs of saturated chains differing by a single flip across a 2-face of $Q$. In this case the vertices of $\mathcal{M}(Q,\mathcal{O})$ are the maximal chains of $P$.
\end{defn}

The arguments below are inspired by the proof of \cite[Thm.~2.1]{AER} in the setting of simple polytopes, but some of our adaptations to the directionally simple setting are a bit delicate. This more general setting is important for application to BCFW  bridge decompositions, however, as Bruhat interval polytopes $Q_w$ of Grassmannian permutations $w$ are almost never simple (\cite[Thm.~1.6]{GG} implies a pattern avoidance characterization for such $w$). We have retained notation from \cite{AER} when possible.

\begin{figure} 
\begin{tikzpicture}[scale=0.4]

\draw (7,0)--(4,1.7)--(2,4)--(0,7)--(1,10)--(3.5,12)--(7,13); 
\draw [fill] (4,1.7) circle [radius=0.05];
\draw [fill] (3.5,12) circle [radius=0.05];
\draw [fill] (10,1.7) circle [radius=0.05];
\draw [fill] (10,12) circle [radius=0.05];

\draw (4,1.7)--(5,5)--(8,7)--(10,4.5)--(10,1.7);

\draw [fill] (5,5) circle [radius=0.05];
\draw [fill] (8,7) circle [radius=0.05];
\draw [fill] (10,4.5) circle [radius=0.05];
%\draw [fill] (10,12) circle [radius=0.05];

\draw (5,5)--(5,7)--(4,8)--(2.5,7)--(2,4);

\draw [fill] (5,7) circle [radius=0.05];
\draw [fill] (4,8) circle [radius=0.05];
\draw [fill] (2.5,7) circle [radius=0.05];
%\draw [fill] (10,12) circle [radius=0.05];

\draw (2.5,7)--(1,10);

%\draw(5,7)--(1,10);
%\draw(7,4)--(8,7)--(7,10);
\draw(13,10)--(10,12)--(7,13);
\draw(7,0)--(10,1.7)--(13,4)--(14,7)--(13,10); 
%\draw [dashed] (2,4)--(10,7)--(9,10)--(7,13);
\draw [dashed] (8,7)--(7,13);
\draw [dashed] (4,8)--(7,13);
%\draw [dashed] (14,7)--(7,13);

\draw[thick,color=blue] (7,0)--(4,1.7)--(2,4)--(0,7)--(1,10)--(3.5,12)--(7,13);
%\draw[thick,color=green,-latex] (7,1)--(4,2)--(2,4)--(0,7)--(1,10)--(3.5,12)--(7,13);

\draw (2,4)--(2.5,7)--(1,10);
\draw[thick,color=green] (7,0.1)--(4,1.8)--(5,5)--(8,7);
\draw[dashed,color=green] (8,7)--(7,13);
\draw[dashed,color=orange] (8.05,7)--(7.05,13);
\draw[thick,color=orange] (7,0.1)--(10,1.8)--(10,4.5)--(8,7);
\draw[thick,color=purple] (7,0)--(10,1.7)--(13,4)--(14,7)--(13,10)--(10,12)--(7,13);

\node [below] at (7,0){\scriptsize $x$}; 
\node [above] at (7,13){\scriptsize $x'$}; 
%\node [below left] at (7,4){\scriptsize $1324$}; 
%\node [below right] at (13,4){\scriptsize $1243$}; 
\node [below left,color=orange] at (10.5,7){\scriptsize $\gamma_2^F$}; 
\node [below left,color=green] at (7.8,6.2){\scriptsize $\gamma_1^F$}; 
\node [below right] at (9.8,1.7){\scriptsize $a_2$}; 
\node [above right] at (7.7,9){\scriptsize $p$}; 
\node [below right] at (3,1.7){\scriptsize $a_1$}; 
\node [above left,color=blue] at (0.5,4.3){\scriptsize $\gamma_1$}; 
\node [above left,color=purple] at (15.3,4.3){\scriptsize $\gamma_2$}; 
\node [above right] at (6.5,2){\scriptsize $F$}; 
\node [above right] at (2.8,4.4){\scriptsize $F'$}; 
\node [above] at (1.3,6.4){\scriptsize $F''$};

\draw [fill, ] (7,0) circle [radius=0.05]; 
\draw [fill] (2,4) circle [radius=0.05]; 
%\draw [fill] (7,4) circle [radius=0.05]; 
\draw [fill] (13,4) circle [radius=0.05]; 
\draw [fill] (0,7) circle [radius=0.05]; 
%\draw [fill] (5,7) circle [radius=0.05]; 
%\draw [fill] (8,7) circle [radius=0.05]; 
%\draw [fill] (10,7) circle [radius=0.05];
\draw [fill] (14,7) circle [radius=0.05];
\draw [fill, ] (1,10) circle [radius=0.05]; 
%\draw [fill] (7,10) circle [radius=0.05]; 
%\draw [fill] (9,10) circle [radius=0.05]; 
\draw [fill] (13,10) circle [radius=0.05]; 
\draw [fill] (7,13) circle [radius=0.05]; 

\end{tikzpicture}
\caption{An example showing some of the chains  in the path $\gamma_1 * \gamma_2$ constructed in \Cref{lem:path-exists}. The first flip is across the face $F''$.}
\label{fig:2-face-first-move}
\end{figure}

\begin{lem}\label{lem:path-exists}
Let $Q$ be an $\mathcal{O}$-directionally simple polytope such that $\mathcal{O}$ is the Hasse diagram of a lattice $L$. Then, for any $u \leq_L v$, and any chains $\gamma_1,\gamma_2 \in \mathcal{M}(u,v,\mathcal{O})$, there exists a path from $\gamma_1$ to $\gamma_2$ such that all intermediate chains $\gamma$ contain all vertices in $\gamma_1 \cap \gamma_2$. 
\end{lem}
\begin{proof}  
We construct such a path, denoted $\gamma_1 * \gamma_2$, as follows.

Suppose that $\gamma_1 \neq \gamma_2$ and let $x$ be the lowest element of $\gamma_1 \cap \gamma_2$ such that $x$ is covered by distinct  elements in $\gamma_1$ and $\gamma_2$. Let $x'$ be the lowest element of $\gamma_{1} \cap \gamma_{2}$ that is greater than $x$; such an element exists, since $\gamma_{1}$ and $\gamma_{2}$ have the same sink. Let $e_1$ (resp. $e_2$)  be the edge upward from $x$ in $\gamma_1$  (resp. $\gamma_2$ ). By directional simplicity, there exists a 2-face $F$ of $Q$ containing  $e_1$ and $e_2$. By \cite[Thms.~4.6 \& 5.13]{He}, $\sink(F)=a_1 \vee a_2$, where $a_1,a_2$ are the vertices connected to $x$ by $e_1,e_2$. Since $x'\geq a_1, a_2$, the lattice property implies that there exists a saturated chain $p$ from $\sink(F)$ to $\sink(Q)$ passing through $x'$.  If $\sink(F) \in \gamma_1$, then choose $p$ to coincide with $\gamma_1$ on $[\sink(F),\sink(Q)]$. In any event, choose $p$ to coincide with $\gamma_1$ on $[x',\sink(Q)]$.  

Let $\gamma_1^F$ be the chain that coincides with $\gamma_1$  on $[\srce(Q),a_1]$, 
then follows the unique path in  the boundary of $F$ proceeding upward from $a_1$ to $\sink(F)$, then coincides with $p$ on $[\sink(F),\sink(Q)]$.  
Likewise, let $\gamma_2^F$ be the maximal chain that coincides with $\gamma_2$ on $[\srce(Q),a_2]$, then follows the unique  path  in  the boundary of $F$ proceeding upward from  $a_2$  to $\sink(F)$, then coincides with $p$ on $[\sink(F),\sink(Q)]$.   

By induction on the length of the longest path from $x$ to $\sink(Q)$, we may assume that the paths $\gamma_1 * \gamma_1^F$ and $\gamma_2^F * \gamma_2$ have already been specified, and use these in our construction of the path $\gamma_1 * \gamma_2$; specifically, we use that $\gamma_i$ coincides with $\gamma_i^F$ on $[\srce(Q),a_i]$ for $i=1,2$. Now let $\gamma_1 * \gamma_2$ be the path in $\mathcal{M}(u,v,\mathcal{O})$ determined by the following sequence of flips: first apply the inductively-defined series of flips in $\gamma_1 * \gamma_1^F$; then apply the flip across $F$ from $\gamma_1^F$ to $\gamma_2^F$; finally, apply the inductively-defined series of flips in $\gamma_2^F*\gamma_2$. Note that this construction depends on the choices of the path $p$ in each iteration. It is clear by the choices of $x,x'$ that all intermediate chains $\gamma$ in $\gamma_1 * \gamma_2$ contain all vertices of $\gamma_1 \cap \gamma_2$.
\end{proof}

The following special feature of the path $\gamma_1 * \gamma_2$ constructed in the proof of \Cref{lem:path-exists} will be useful going forward.

\begin{cor}\label{cor:sink-lemma}
In the notation of \Cref{lem:path-exists} and its proof, if $\sink(F)\in \gamma_1$, then all elements of $\gamma_1 * \gamma_2$ other than $\gamma_1$ contain the edge $e_2$.
\end{cor}
\begin{proof}
Our choice of path $p$ in the case when $\sink (F)\in \gamma_1$ ensures in this case that $\gamma_1=\gamma_1^F$ and therefore that the earliest flip in $\gamma_1 * \gamma_2$ is the flip across $F$ producing $\gamma_2^F$. Since $\gamma_2^F$ and $\gamma_2$ both contain $e_2$, the lemma guarantees that so do all subsequent elements of $\gamma_1 * \gamma_2$. 
\end{proof}

\Cref{lem:path-exists} implies in particular that $\mathcal{M}(u,v,\mathcal{O})$ is connected. Our goal in the remainder of the section is to study the higher vertex connectivity of $\mathcal{M}(u,v,\mathcal{O})$.

\begin{lem}\label{lem:unique-sequence}
Let $Q$ be an $\mathcal{O}$-directionally simple polytope such that $\mathcal{O}$ is the Hasse diagram of a lattice $L$.   Suppose  $\gamma $ and $\gamma_1$ are saturated chains in an interval $[u,v]_L$ that include distinct cover relations upward from $u$.  Let $\gamma_2$ be any node, distinct from $\gamma $, 
on the path $\gamma * \gamma_1$.   Denote by $e_1$  the lowest edge in $\gamma_1$, and denote by  $e_2$ the lowest edge in $\gamma_2$ that is  not in $\gamma$.  
  
Then there exists an  alternating sequence $(\epsilon_0,F_1,\epsilon_1,F_2,\dots ,F_r,\epsilon_r)$ such that  $\epsilon_0 = e_1$, $\epsilon_r = e_2$, the edges $\epsilon_0,\dots ,\epsilon_r$ are distinct, 
each  $\epsilon_i$ is an edge of $Q$ whose lower vertex $x_i$ is in $\gamma$,  and
 each $F_i$ is a 2-face of $Q$ containing  $\epsilon_{i-1}, \epsilon_i$ and  all edges and vertices of $\gamma$ between $x_{i-1}$ and $x_i$.
\end{lem}

\begin{proof}
By directional simplicity we may choose $F_1$ to be the unique 2-face containing both $e_1$ and the lowest edge in $\gamma$, which are distinct by hypothesis. If $\sink (F_1) \in \gamma$, then  $e_2 = e_1$ by \Cref{cor:sink-lemma}.  In this case our desired alternating sequence has $r=0$ and just consists of $(\epsilon_0)$, allowing us to assume $\sink (F_1)\not\in \gamma $ henceforth.
  
 If $e_2$ is an edge in $F_1$, let $\epsilon_1=e_2$, in which case $r=1$ and we are done. Otherwise, define $\epsilon_1$ to be the lowest edge not contained in $\gamma$ that is in the path, in the boundary of $F_1$, from $\srce(F_1)$ to $\sink(F_1)$ that does not contain $e_1$. Such an edge $\epsilon_1$ exists, since otherwise $\sink(F_1)$ would be in $\gamma$. For $\epsilon_1 \ne e_2$ let $F_2$ be the unique 2-face containing both $\epsilon_1$ and the edge in $\gamma$ that proceeds upward from a vertex in  $\epsilon_1$.  Repeat in this manner until reaching a 2-face that contains $e_2$; this process terminates by virtue of the construction in \Cref{lem:path-exists}. 
\end{proof}

\begin{figure} 
\begin{tikzpicture}[scale=0.4]

\draw [color=blue] (7,0)--(4,1.7)--(2,4)--(0,7)--(1,10)--(3.5,12)--(7,13); 
\draw [fill] (4,1.7) circle [radius=0.05];
\draw [fill] (3.5,12) circle [radius=0.05];
\draw [fill] (10,1.7) circle [radius=0.05];
\draw [fill] (10,12) circle [radius=0.05];

\draw (4,1.7)--(5,5)--(8,7)--(10,4.5)--(10,1.7);

\draw [fill] (5,5) circle [radius=0.05];
\draw [fill] (8,7) circle [radius=0.05];
\draw [fill] (10,4.5) circle [radius=0.05];
%\draw [fill] (10,12) circle [radius=0.05];

\draw (5,5)--(5,7)--(4,8)--(2.5,7)--(2,4);

\draw [fill] (5,7) circle [radius=0.05];
\draw [fill] (4,8) circle [radius=0.05];
\draw [fill] (2.5,7) circle [radius=0.05];
%\draw [fill] (10,12) circle [radius=0.05];

\draw (2.5,7)--(1,10);

%\draw(5,7)--(1,10);
%\draw(7,4)--(8,7)--(7,10);
\draw(13,10)--(10,12)--(7,13);
\draw(7,0)--(10,1.7)--(13,4)--(14,7)--(13,10); 
%\draw [dashed] (2,4)--(10,7)--(9,10)--(7,13);
\draw [dashed] (8,7)--(7,13);
\draw [dashed] (4,8)--(7,13);
%\draw [dashed] (14,7)--(7,13);

%\draw[thick,color=blue] (7,0)--(4,1.7)--(2,4)--(0,7)--(1,10)--(3.5,12)--(7,13);
%\draw[thick,color=green,-latex] (7,1)--(4,2)--(2,4)--(0,7)--(1,10)--(3.5,12)--(7,13);

\draw[thick,color=red] (2,4)--(2.5,7);
\draw[thick,color=red] (4,1.7)--(5,5);
\draw[thick,color=red] (7,0)--(10,1.7);
\draw[thick,color=green](2.5,7)--(4,8)--(7,13);

%\node [below] at (7,1){\scriptsize $1234$}; 
%\node [below left] at (2,4){\scriptsize $2134$}; 
%\node [below left] at (7,4){\scriptsize $1324$}; 
%\node [below right] at (13,4){\scriptsize $1243$}; 
%\node [below left] at (0,7){\scriptsize $2314$}; 
%\node [below left] at (5,7){\scriptsize $3124$}; 
\node [below right] at (8.4,0.9){\scriptsize $\epsilon_0$}; 
\node [above right] at (4.3,2.4){\scriptsize $\epsilon_1$}; 
\node [below right] at (2,5.5){\scriptsize $\epsilon_2$}; 
\node [above left,color=blue] at (0.5,4.3){\scriptsize $\gamma $}; 
\node [above left] at (15.7,4.3){\scriptsize $\gamma_1$}; 
\node [above left,color=green] at (6.5,8.8){\scriptsize $\gamma_2$};
\node [above right] at (6.5,2.5){\scriptsize $F_1$}; 
\node [above right] at (2.9,5){\scriptsize $F_2$}; 
%\node [above] at (1.3,6.4){\scriptsize $F''$}; 

\draw [fill, ] (7,0) circle [radius=0.05]; 
\draw [fill] (2,4) circle [radius=0.05]; 
%\draw [fill] (7,4) circle [radius=0.05]; 
\draw [fill] (13,4) circle [radius=0.05]; 
\draw [fill] (0,7) circle [radius=0.05]; 
%\draw [fill] (5,7) circle [radius=0.05]; 
%\draw [fill] (8,7) circle [radius=0.05]; 
%\draw [fill] (10,7) circle [radius=0.05];
\draw [fill] (14,7) circle [radius=0.05];
\draw [fill, ] (1,10) circle [radius=0.05]; 
%\draw [fill] (7,10) circle [radius=0.05]; 
%\draw [fill] (9,10) circle [radius=0.05]; 
\draw [fill] (13,10) circle [radius=0.05]; 
\draw [fill] (7,13) circle [radius=0.05]; 

\end{tikzpicture}
\caption{An example of a sequence $(\epsilon_0,F_1,\epsilon_1,\dots ,F_r,\epsilon_r)$ from \Cref{lem:unique-sequence} with  $r=2$.}
\label{fig:edge-face-sequence}
\end{figure}

\begin{lem}\label{lem:what-determines}
 The sequence $(\epsilon_0,F_1,\epsilon_1,F_2,\dots ,F_r,\epsilon_r)$ from  
 \Cref{lem:unique-sequence}
  is uniquely determined by $\gamma $ and $e_2$. In particular, $\gamma$ and $e_2$ determine $e_1$. 
\end{lem}
\begin{proof}
If $e_2$ contains the vertex $u$, then $r=0$ and so the sequence is $(\epsilon_0)$ and $e_1=e_2$, so the claim is satisfied. Otherwise, by the requirements for $F_r$  in  \Cref{lem:unique-sequence}, the face $F_r$ includes both the edge $\epsilon_r=e_2$ and an edge in $\gamma $ leading upward to $\epsilon_r$.  We showed such an $F_r$ exists in the proof of \Cref{lem:unique-sequence}. There can only be one 2-face in a polytope containing a specified pair of edges, so this 2-face is uniquely determined.  By virtue of  the edge $\epsilon_{r-1}$ having its lower vertex  in $\gamma $ and the entire path from $x_{r-1}$ to $x_r$ also being  in $\gamma $, 
the edge $\epsilon_{r-1}$ must be the unique  lowest edge  not contained in $\gamma $ that is in the boundary of $F_r$.  We likewise deduce  that  $F_{r-1}$ is  uniquely determined as the only 2-face that includes both $\epsilon_{r-1}$ and the edge in $\gamma $ leading up to  $\epsilon_{r-1}$.  Repeating in this manner shows that $F_1$ is likewise uniquely determined, and $F_1$  uniquely determines the edge $\epsilon_0$  in the  boundary of $F_1$ that proceeds upward from the source of $F_1$ and is not in $\gamma $.
\end{proof}

\begin{lem}\label{lem:vertex-disjoint}
Let $Q$ be an $\mathcal{O}$-directionally simple polytope such that $\mathcal{O}$ is the Hasse diagram of a lattice $L$.   If $\gamma, \gamma_1$, and $\gamma_1'$ are saturated chains in an interval $[u,v]_L$ using distinct edges upward from $u$, then $(\gamma * \gamma_1) \cap (\gamma * \gamma_1')=\{\gamma\}$.
\end{lem}

\begin{proof}
Suppose that $\gamma \neq \gamma '' \in (\gamma * \gamma_1)\cap (\gamma * \gamma_1')$.
Since $\gamma_1$ and $\gamma_1'$ have distinct lowest edges, $\gamma ''$ must have the same lowest edge as $\gamma$. Letting $F$ (resp. $F'$) denote the unique 2-face containing the lowest edges in $\gamma $ and $\gamma_1 $ (resp. $\gamma $ and $\gamma_1'$), it follows that $\gamma''$ must occur prior to the flip across $F$ (resp. $F'$) within $\gamma * \gamma_1$ (resp. $\gamma * \gamma_1'$). If $\sink (F)\in \gamma $ this would contradict \Cref{cor:sink-lemma}. Thus it suffices to consider the case
$\sink (F) \not \in \gamma $ (resp. $\sink (F') \not \in \gamma $). 
This implies that $r\ge 1$ in  the sequence $(\epsilon_0,F_1,\dots ,\epsilon_r)$ determined according to \Cref{lem:unique-sequence} by $\gamma$ and the lowest edge in $\gamma ''$ not in $\gamma$. This allows us to  assume $r>0$ in the remainder of this proof.

Consider any node $\gamma_2$ on the path $\gamma*\gamma_1$ and any node $\gamma_2'$ on the path $\gamma * \gamma_1'$.  Let $e_2$ be the lowest edge in $\gamma_2$ not shared with $\gamma $, and let $e_2'$ be the lowest edge in $\gamma_2' $ not shared with $\gamma $.   To prove $\gamma_2 \ne \gamma_2'$, it suffices to prove $e_2 \ne e_2'$.  If $e_2 = e_2'$, then \Cref{lem:what-determines} shows that $\gamma $ and $e_2$ would give rise to the same  alternating sequence of edges and 2-faces that  $\gamma $ and $e_2'$ would produce.  In particular, they would give rise to the same edges upward from $u$, contradicting $\gamma_1$ and $\gamma_1'$ having  distinct edges upward from $u$.
\end{proof}

We are now ready to prove \Cref{thm:intro-moves}, our second main theorem.

\begin{proof}[Proof of \Cref{thm:intro-moves}]
The case with $a=1$ is handled by \Cref{lem:path-exists}. The proof of the case with $a \geq 2$ is via the same argument as in \cite{AER}, with the dimension $d$ of a simple polytope replaced by the number $a$ of atoms of $[u,v]_L$ throughout, using \Cref{lem:path-exists,lem:unique-sequence,lem:what-determines,lem:vertex-disjoint} to justify various claims under our modified hypotheses. We recall this argument now, as it applies in our setting.
     
Let $\Gamma$ be any subset of size at most $a-2$ of the set of vertices of $\mathcal{M}(u,v,\mathcal{O})$.  Let $E$ be the set of atoms of $[u,v]_L$ that are contained in chains from $\Gamma$. Let $M$ be the induced subgraph of $\mathcal{M}(u,v,\mathcal{O})\setminus \Gamma $ whose vertices are those chains $\gamma$ that do not use any cover relation from $E$. Notice that $|M|\ge 2$, since at least 2 atoms of $[u,v]_L$ lie outside $E$ and there is at least one saturated chain containing each such atom.

Given any $\gamma_1,\gamma_2\in M$, observe that no vertex on the path $\gamma_1 * \gamma_2$ uses any atom from $E$, since all $\gamma \in \gamma_1 * \gamma_2$ use an atom used by either $\gamma_1$ or $\gamma_2$.  Thus $M$ is connected.
     
It remains to prove that for any chain $\gamma \not \in \Gamma $ whose lowest cover relation $e$ uses an atom from $E$ there is a path from $\gamma $ to an element of $M$ that stays within  $\mathcal{M}(u,v,\mathcal{O}) \setminus \Gamma$.  Let $k$ be the size of the subset of $\Gamma $ consisting of chains using the edge $e$.  The number of cover relations upward from $u$ in $[u,v]_L$ that are not in $E$ equals  $a - |E| \ge a - |\Gamma | + (k-1)  \ge k+1$.  This implies the existence of $k+1$ saturated chains in $M$ all involving distinct cover relations upward from $u$.  By \Cref{lem:vertex-disjoint}, the paths in $\mathcal{M}(u,v,\mathcal{O})$ from $\gamma $ to these $k+1$ saturated chains are vertex disjoint except for sharing the initial node $\gamma $.  Thus at least one of these paths  avoids all $k$ elements of $\Gamma $ having initial edge $e$.  By the construction from \Cref{lem:path-exists}, this path stays within $\mathcal{M}(u,v,\mathcal{O}) \setminus \Gamma $.  
\end{proof}

\section*{Acknowledgements}

We are grateful to Richard Stanley for suggesting the investigation of the topology of Bruhat interval polytope posets. We also thank Melissa Sherman-Bennett and Lauren Williams for helpful conversations about BCFW bridge decompositions. We also thank  the anonymous referee for very helpful feedback which we have incorporated throughout the paper.

\bibliographystyle{plain}
\bibliography{Bruhat-polytope-topology-final}
\end{document}